\documentclass[a4paper]{amsart}

\usepackage{amsfonts}
\usepackage{verbatim}
\usepackage{amssymb}
\usepackage{amscd}
\usepackage{amsmath}
\usepackage{latexsym}
\usepackage{framed, color}

\usepackage[dvips]{graphicx}
\usepackage{psfrag}

\newtheorem{thm}{Theorem}[section]
\newtheorem{cor}[thm]{Corollary}

\newtheorem{lem}[thm]{Lemma}

\theoremstyle{definition}
\newtheorem{dfn}[thm]{Definition}

\theoremstyle{remark}

\newtheorem{ex}[thm]{Example}

\makeatletter
 
 \@addtoreset{equation}{section}
\makeatother

\newcommand{\R}{\mathbb{R}}

\newcommand{\Q}{\mathbb{Q}}

\newcommand{\Z}{\mathbb{Z}}

\newcommand{\Int}{\operatorname{Int}}
\newcommand{\id}{\operatorname{id}}

\def\spmapright#1{\smash{%
 \mathop{\hbox to 1.3cm{\rightarrowfill}}
  \limits^{#1}}}
\def\spmapleft#1{\smash{%
 \mathop{\hbox to 1.3cm{\leftarrowfill}}
  \limits^{#1}}}

\title[Round fold maps on $3$--manifolds]
{Round fold maps on $3$--manifolds}
\author{Naoki Kitazawa and Osamu Saeki} 
\address{Institute of Mathematics for Industry,
Kyushu University,
Motooka 744, Nishi-ku, Fukuoka 819-0395, Japan}

\email{n-kitazawa@imi.kyushu-u.ac.jp}

\email{saeki@imi.kyushu-u.ac.jp}
\date{\today}


\begin{document}
\begin{abstract}
We show that a closed orientable $3$--dimensional
manifold admits
a round fold map into the plane, i.e.\ a fold
map whose critical value set consists of disjoint simple closed curves
isotopic to concentric circles, if and only if it is
a graph manifold, generalizing the characterization
for simple stable maps into the plane. Furthermore, we also
give a characterization of
closed orientable graph manifolds that admit directed
round fold maps into the plane, i.e.\ round fold maps
such that the number of regular fiber components of a
regular value increases
toward the central region in the plane.
\end{abstract}

\maketitle

\section{Introduction}\label{section1}

Let $M$ be a smooth closed manifold of dimension
$\geq 2$. It is known that if a smooth map $f : M \to
\R^2$ is generic enough, then it has only fold
and cusps as its singularities \cite{L0, L1, Wh}.
Furthermore, if $M$ has even Euler characteristic
(e.g.\ if $\dim{M}$ is odd), then the cusps can
be eliminated by homotopy. In particular,
every smooth closed orientable $3$--dimensional manifold
admits a smooth map into $\R^2$ with only fold singularities,
i.e.\ a \emph{fold map}.

In \cite{Sa1, Sa2}, the second author
considered the following smaller class
of generic smooth maps. A fold map $f : M \to \R^2$
on a smooth closed orientable $3$--dimensional manifold $M$
is a {\em simple stable map} if for every $q \in \R^2$, each component 
of $f^{-1}(q)$ contains at most one singular point and 
$f|_{S(f)}$ is an immersion with normal crossings,
where $S(f) (\subset M)$ denotes the set of
singular points of $f$. 
Note that if $f$ is a fold
map, then $S(f)$ is a regular closed submanifold of $M$.
In particular, if $f|_{S(f)}$ is an embedding,
then $f$ is a simple stable map.
In \cite{Sa1},
it has been proved that for a smooth closed orientable $3$--dimensional 
manifold $M$, the following three are equivalent to each other:
\begin{enumerate}
\item $M$ admits a fold
map $f : M \to \R^2$ such that $f|_{S(f)}$ is an embedding,
\item  $M$ admits a simple stable map into $\R^2$,
\item $M$ is a graph manifold, i.e. it 
is a finite union of $S^1$--bundles over compact surfaces
attached along their torus boundaries.
\end{enumerate}
Thus, for example, if $M$ is hyperbolic, then
$M$ never admits such a fold map.

On the other hand, the first author introduced the notion
of a round fold map \cite{Ki1, Ki2}: a smooth
map $f : M \to \R^2$ is a \emph{round fold map}
if it is a fold map and $f|_{S(f)}$ is an embedding
onto the disjoint union of some concentric circles in $\R^2$ (for details,
see \S\ref{section2}). As has been studied by the
first author, round fold maps have various nice properties.

The first main result of this paper is Theorem~\ref{thm1}, which
states that every graph $3$--manifold admits a round fold
map into $\R^2$. This generalizes the characterization result
obtained in \cite{Sa1} for simple stable maps mentioned above.

It is not difficult to observe that if $f : M \to \R^2$
is a round fold map of a closed orientable $3$--dimensional
manifold, then the number of components of the
fiber over a regular value changes exactly by one when
the regular value crosses the critical value set. We can thus
put a normal orientation to each component of the critical
value set in such a way that the orientation points
in the direction that increases the number of components
of a regular fiber.
Then, a round fold map is said to be \emph{directed}
if all the circles in the critical value set are directed
inward. The second main result of this paper (Theorem~\ref{thm2})
characterizes those graph $3$--manifolds which admit
directed round fold maps. It will turn out that the class
is strictly smaller than that of closed orientable
graph $3$--manifolds.

The paper is organized as follows. In \S\ref{section2},
we prepare several definitions and a lemma concerning
round fold maps and graph $3$--manifolds, and state
our main theorems. We also give an observation
on fibered links or open book structures
associated with round fold maps and give some examples.
In \S\ref{section3}, we prove the
main theorems. Basically, we will follow the proof 
given in \cite[Theorem~3.1]{Sa1}: however, in some steps
we need to modify the strategy for the constructions
of round fold maps.
In \S\ref{section4}, we give some corollaries and show
that the class of $3$--manifolds that admit
directed round fold maps is strictly smaller than
that of all graph $3$--manifolds, using results obtained
in \cite{DH, N}.

Throughout the paper, 
all manifolds and maps between them are smooth
of class $C^\infty$ unless otherwise specified. 
For a space $X$, 
$\id_X$ denotes the identity map of $X$. 
The symbol ``$\cong$'' denotes a diffeomorphism between
smooth manifolds.

\section{Round fold maps}\label{section2}

Let $M$ be a closed orientable $3$--dimensional manifold and
$f : M \to \R^2$ a smooth map.

\begin{dfn}
A point $p \in M$ is a \emph{singular point} of $f$
if the rank of the differential $df_p : T_pM \to T_{f(p)}\R^2$
is strictly less than two. We denote by $S(f)$ the set
of all singular points of $f$.
A point $p \in S(f)$ is a \emph{definite fold point} (resp.\ 
an \emph{indefinite fold point}, or a \emph{cusp point})
if $f$ is represented by the map
$$(u, x, y) \mapsto (u, x^2 + y^2) 
\mbox{ (resp.\ $(u, x^2 - y^2)$, or $(u, y^2+ux-x^3)$)}$$
around the origin with respect to certain local coordinates around $p$
and $f(p)$. 
We call a point $p \in S(f)$ a \emph{fold point}
if it is a definite or an indefinite fold point.
A smooth map $f : M \to \R^2$ is called a \emph{fold map}
if it has only fold points as its singular points.
Note that then $S(f)$ is a closed $1$--dimensional submanifold of $M$
and that $f|_{S(f)}$ is an immersion.
\end{dfn}

\begin{dfn}
Let $C$ be a finite disjoint union of simple closed curves in $\R^2$.
We say that $C$ is \emph{concentric} if for every pair $c_0, c_1$
of distinct components of $C$, exactly one of them, say $c_i$, is contained
in the bounded region of $\R^2 \setminus c_{1-i}$ (see Fig.~\ref{fig1}).
(In this case, we say that $c_i$ (or $c_{1-i}$)
is an \emph{inner component} (resp.\ an \emph{outer component})
with respect to $c_{1-i}$ (resp.\ $c_i$).)
In other words, $C$ is isotopic to a set of concentric circles in $\R^2$.
\end{dfn}

\begin{dfn}
We say that a smooth map $f : M \to \R^2$ of a closed $3$--dimensional
manifold $M$ into the plane is a \emph{round fold map}
if it is a fold map and $f|_{S(f)}$ is an embedding onto a concentric
family of simple closed curves.
Note that a round fold map is a simple stable map in the sense of
\cite{Sa1, Sa2}. Note also that the outermost circle component
of $f(S(f))$ consists of the images of definite fold points.
\end{dfn}

\begin{figure}[h]
\centering
\psfrag{C}{$C$}
\psfrag{ci}{$c_i$}
\psfrag{c1}{$c_{1-i}$}
\includegraphics[width=0.9\linewidth,height=0.3\textheight,
keepaspectratio]{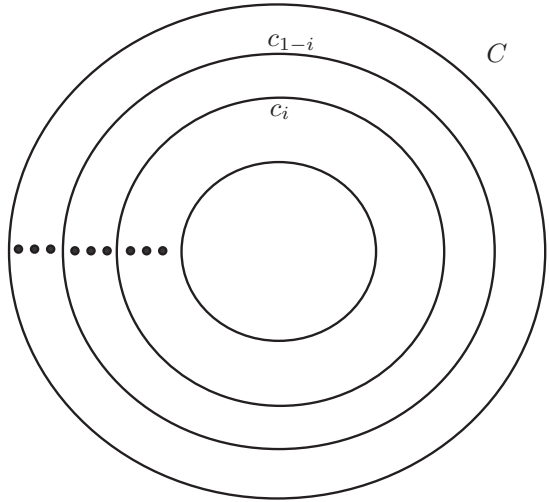}
\caption{Family of concentric simple closed curves in $\R^2$}
\label{fig1}
\end{figure}

In the following, $A$ denotes the annulus $S^1 \times
[-1, 1]$, and $P$ denotes the compact surface obtained from
the $2$--sphere by removing three open disks: 
in other words, $P$ is a \emph{pair of pants}.

Let $f : M \to \R^2$ be a round fold of a closed
orientable $3$--dimensional manifold $M$.
For a component $c$ of $f(S(f))$, take a small
arc $\alpha \cong [-1, 1]$ in $\R^2$ that intersects $f(S(f))$
exactly at one point in $c$ transversely.
Then, $f^{-1}(\alpha)$ is a compact surface
with boundary $f^{-1}(a) \cup f^{-1}(b)$, which
is diffeomorphic to a finite disjoint union of circles, where
$a$ and $b$ are the end points of $\alpha$.
Furthermore,
$f|_{f^{-1}(\alpha)} : f^{-1}(\alpha) \to \alpha$
can be regarded as a Morse function with exactly
one critical point. As $M$ is orientable, we see that
$f^{-1}(\alpha)$ is diffeomorphic to the union of $D^2$
(or $P$) and a finite number of copies of $A$
(see \cite{Sa04}, for example).
Therefore, the number of components of $f^{-1}(a)$
differs from that of $f^{-1}(b)$ exactly by one. If $f^{-1}(a)$
has more components than $f^{-1}(b)$, then we normally
orient $c$ from $b$ to $a$: otherwise, we orient $c$ from $a$ to $b$.
It is easily shown that this normal orientation is 
independent of the choice of $\alpha$.
In this way, each component of $f(S(f))$ is normally
oriented. If the normal orientation points inward,
then the component is said to be \emph{inward-directed}:
otherwise, \emph{outward-directed}.

\begin{dfn}
Let $f : M \to \R^2$ be a round fold map. 
We say that $f$ is \emph{directed}
if all the components of $f(S(f))$ are inward-directed.
It is easy to see that a round
fold map $f$ is directed if and only if the number of
components of a regular fiber over a point in the
innermost component of $\R^2 \setminus f(S(f))$ coincides with
the number of components of $S(f)$.
\end{dfn}

Let $f : M \to \R^2$ be a round fold map of
a closed connected oriented $3$--dimensional manifold.
In the following, for $r > 0$,
$C_r$ denotes 
the circle of radius $r$ centered at the origin in $\R^2$.
We may assume that 
$$f(S(f)) = \bigcup_{i=1}^t C_i$$
for some $t \geq 1$ by composing a diffeomorphism of $\R^2$
isotopic to the identity if necessary.
Set $L = f^{-1}(0)$, which is an oriented
link in $M$ if it is not empty.
Let $D$ be the closed disk centered at the origin with
radius $1/2$. 
Then, $f^{-1}(D)$ is diffeomorphic to
$L \times D$, which can be identified with
a tubular neighborhood $N(L)$ of $L$ in $M$. Furthermore, the composition
$\varphi = \pi \circ f : M \setminus \Int{N(L)} \to S^1$
is a submersion, where $\pi : \R^2 \setminus \Int{D}
\to S^1$ is the standard radial projection and
$\varphi|_{\partial N(L)} : \partial N(L) = L \times \partial D
\to S^1$ corresponds to the projection to the second factor
followed by a scalar multiplication. Hence, $\varphi$
is a smooth fiber bundle and $L$ is a fibered link.
(In other words, $M$ admits an open book structure
with binding $L$.)
The fiber (or the page) is identified with $F = f^{-1}(J)$, where
$$J = [1/2, t+1] \times \{0\} \subset \R^2,$$
and it is a compact oriented surface.
Note that $g = f|_{F} : F \to J$ is
a Morse function with exactly $t$ critical points
and that a monodromy diffeomorphism of the fibration over $S^1$
can be chosen so that it preserves $g$.

Note that all these arguments work even when
$L = \emptyset$. In this case, $F$ is a closed orientable
surface and $M$ is the total
space of an $F$--bundle over $S^1$.

Conversely, if we have a compact orientable surface
$F$, a Morse function $g : F \to [1/2, t+1]$
such that $g(\partial F) = 1/2$ and that $g$ has no
critical point near the boundary,
and a diffeomorphism $h : F \to F$ which is the
identity on the boundary and which satisfies $g \circ h = g$,
then we can construct a round fold map
$f : M \to \R^2$ in such a way that $M$ is the union
of $\partial F \times D^2$ and the total space
of the $F$--bundle over $S^1$ with geometric monodromy $h$.

\begin{ex}
Let $F$ be a compact connected orientable
surface with $\partial F \neq \emptyset$.
Let us consider the identity diffeomorphism
as the geometric monodromy in the above
construction. Then, we see that the source $3$--manifold
$M$ of the round fold map is diffeomorphic to
$(\partial F \times D^2) \cup (F \times S^1)
\cong \partial (F \times D^2)$. By using
a handle decomposition argument, we see easily that
$F \times D^2$ is diffeomorphic to $D^4$ or
a boundary connected sum of a finite number of
copies of $S^1 \times D^3$. Therefore,
$M$ is diffeomorphic either to $S^3$ or
to the connected sum of a finite number of
copies of $S^1 \times S^2$.

For example, if we start with the Morse function $g_1 : F_1
\to [1/2, 4]$
as depicted in Fig.~\ref{fig3} (left), then the singular point set
$S(f_1)$ of the resulting round fold map $f_1 : M_1 \to \R^2$
has three components and their images coincide with
$C_1$, $C_2$ and $C_3$.
The first one is outward directed, while the
other two are inward directed.
Therefore, the fold map $f_1$ is not directed. In this
example, $M_1$ is diffeomorphic to $(S^1 \times S^2) \sharp
(S^1 \times S^2)$.

On the other hand, if we start with the Morse
function $g_2 : F_2 \to [1/2, 4]$ as depicted
in Fig.~\ref{fig3} (right), then we get a round fold map
$f_2 : M_2 \to \R^2$ with the same singular values: however,
this round fold map is directed.
We can also show that $M_2$ is again diffeomorphic to
$(S^1 \times S^2) \sharp (S^1 \times S^2)$.
\end{ex}

\begin{figure}[h]
\centering
\psfrag{F}{$F_1$}
\psfrag{g}{$g_1$}
\psfrag{a}{$1/2$}
\psfrag{b}{$1$}
\psfrag{c}{$2$}
\psfrag{d}{$3$}
\psfrag{F2}{$F_2$}
\psfrag{g2}{$g_2$}
\includegraphics[width=\linewidth,height=0.4\textheight,
keepaspectratio]{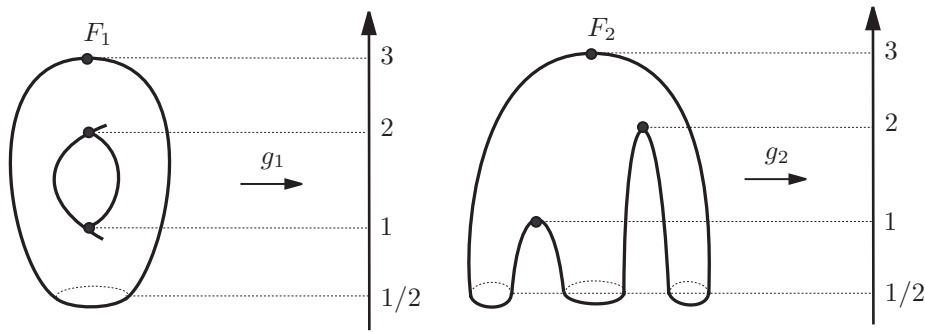}
\caption{Morse functions on surfaces with Euler characteristic $-1$}
\label{fig3}
\end{figure}

\begin{dfn}
Let $M$ be a closed orientable $3$--dimensional
manifold. It is called a \emph{graph manifold}
if it is diffeomorphic to a union of
$S^1$--bundles over compact surfaces attached along their
torus boundaries.
\end{dfn}

\begin{lem}\label{lem1}
Every closed orientable graph $3$--manifold
is diffeomorphic to a union of finite numbers of
copies of $P \times S^1$ and solid tori attached
along their torus boundaries.
\end{lem}

\begin{proof}
It is known that every such $3$--manifold
is diffeomorphic to a union of a finite number
of $S^1$--bundles over compact connected orientable surfaces of genus zero
(for example, see \cite[Lemma~3.3]{Sa1}).
For such a base surface $B$,
if the number of boundary components is greater than
or equal to $4$, then we can decompose $B$ into
a union of a finite number of copies of $P$ attached
along their circle boundaries. If the number of boundary
components is equal to two, then $B$ is diffeomorphic
to the union of $P$ and a disk.
If the surface $B$ has no boundary, then we can decompose it into
two disks. As orientable $S^1$--bundles over $P$ or a disk are always
trivial, the result follows.
\end{proof}

As a consequence, a graph manifold can be represented by
a (multi-)graph, where each vertex corresponds to $P \times S^1$
or a solid torus and each edge corresponds to the gluing
along a pair of boundary components. Note that each gluing
corresponds to an element of the (orientation preserving)
mapping class group of the torus, identified with $SL(2, \Z)$.

Our main results of this paper are as follows.

\begin{thm}\label{thm1}
Let $M$ be a closed orientable $3$--dimensional manifold.
Then, it admits a round fold map into $\R^2$ if and only
if it is a graph manifold.
\end{thm}

In particular, every closed orientable graph $3$--manifold
admits a fibered link. Compare this with \cite{Myers}.

\begin{thm}\label{thm2}
Let $M$ be a closed connected orientable graph $3$--manifold.
Then, it admits a directed round fold map into $\R^2$
if and only if it can be decomposed into
a union of finite numbers of copies of $P \times S^1$
and a solid torus such that the corresponding graph is a tree.
\end{thm}

Note that Theorem~\ref{thm1} generalizes the characterization
for simple stable maps obtained in \cite{Sa1}.



\section{Proofs}\label{section3}

In this section, we prove Theorems~\ref{thm1} and \ref{thm2}.

\begin{proof}[Proof of Theorem~\textup{\ref{thm1}}]
As noted above, a round fold map is a simple stable map.
Therefore, if a closed orientable $3$--dimensional
manifold admits such a map,
then it is necessarily a graph manifold by \cite{Sa1}.

Now, suppose $M$ is a graph manifold.
We will basically follow the proof of \cite[Theorem~3.1]{Sa1}
in order to construct a round fold map $f : M \to \R^2$, except for
the first step, in which a non-singular map
is constructed for each $S^1$--bundle piece in \cite{Sa1} while
we construct a fold map for each piece, as explained below.

By virtue of Lemma~\ref{lem1},
we have disjointly embedded tori $T_1, T_2, \ldots, T_r$
in $M$ such that each of the components $X_1, X_2, \ldots, X_s$
of $M \setminus \sqcup_{i=1}^r 
\Int{N(T_i)}$ is diffeomorphic either to $P \times S^1$ or to 
$D^2 \times S^1$, where $N(T_i)$ denotes a small
tubular neighborhood of $T_i$ in $M$, $1 \leq i \leq r$.
By inserting pieces diffeomorphic to $A \times S^1
\cong T^2 \times [-1, 1]$ if necessary, we may
assume that the decomposition is of a plumbing type
(for details, see \cite[Lemma~3.4]{Sa1}).
Now, each $X_i$ is diffeomorphic either to $P \times S^1$,
$D^2 \times S^1$, or $A \times S^1$.

Take a component $X_j$, $1 \leq j \leq s$. Suppose it is
diffeomorphic to $D^2 \times S^1$.
Let $\delta : D^2 \to [-1, 1]$ be the Morse function
defined by $\delta(x, y) = -x^2-y^2$, where $D^2$
is identified with the unit $2$--disk in $\R^2$
(see Fig.~\ref{fig2}, left).
Then, define $f|_{X_j}$ to be the composition
$$\eta_j \circ (\delta \times \id_{S^1}) \circ \varphi_j: 
X_j \spmapright{\varphi_j} D^2 \times S^1
\spmapright{\delta \times \id_{S^1}} [-1, 1]
\times S^1 \spmapright{\eta_j} \R^2,$$
where $\varphi_j$ is a diffeomorphism and $\eta_j$
is an embedding whose image is a small tubular neighborhood
of the circle of radius $j$ centered at the origin.
We also arrange $\eta_j$ in such a way that
$\eta_j(\{\pm 1\} \times S^1)$ coincides with the
circle of radius $j \pm (1/3)$.

\begin{figure}[h]
\centering
\psfrag{d}{$\delta$}
\psfrag{i}{$\iota$}
\psfrag{1}{$1$}
\psfrag{0}{$0$}
\psfrag{-1}{$-1$}
\psfrag{D}{$D^2$}
\psfrag{P}{$P$}
\includegraphics[width=\linewidth,height=0.3\textheight,
keepaspectratio]{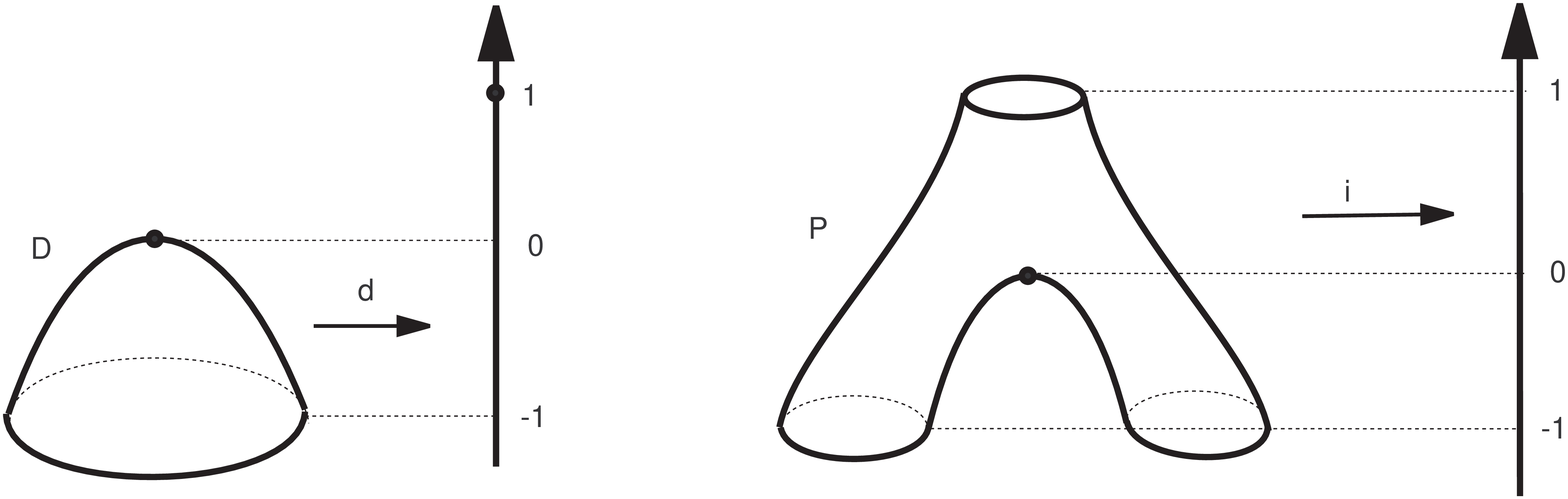}
\caption{Morse functions $\delta$ and $\iota$}
\label{fig2}
\end{figure}

Suppose $X_j$ is diffeomorphic to $P \times S^1$.
We define $f|_{X_j}$ by the composition
$$\eta_j \circ (\iota \times \id_{S^1}) \circ \varphi_j: 
X_j \spmapright{\varphi_j} P \times S^1
\spmapright{\iota \times \id_{S^1}} [-1, 1]
\times S^1 \spmapright{\eta_j} \R^2,$$
where $\varphi_j$ is a diffeomorphism,
$\iota : P \to [-1, 1]$ is the standard Morse
function with exactly one saddle point as depicted in 
Fig.~\ref{fig2} (right), and $\eta_j$
is an embedding as described in the previous paragraph.

Now, suppose $X_j$ is diffeomorphic to $A \times S^1$.
In this case, we define 
$f|_{X_j}$ by the composition
$$\eta_j \circ (\rho \times \id_{S^1}) \circ \varphi_j: 
X_j \spmapright{\varphi_j} A \times S^1
\spmapright{\rho \times \id_{S^1}} [-1, 1]
\times S^1 \spmapright{\eta_j} \R^2,$$
where $\varphi_j$ is a diffeomorphism,
$\rho : A \cong S^1 \times [-1, 1] \to [-1, 1]$ is the 
projection to the second factor, and $\eta_j$
is an embedding as described above.

Now, the map $f|_{\sqcup_{j=1}^s X_j}$ has only fold
singular points, and its restriction to the singular point
set is an embedding onto a concentric family of circles in $\R^2$.
Then, we can extend the map to get a round fold map
$f : M \to \R^2$ by the same procedure as described
in \cite[Proof of Theorem~3.1]{Sa1}.
Note that in the proof there, a simple stable map
into $S^2$ is first constructed: however, in our case,
we can directly construct a map into $\R^2$ as the
singular value set of $f|_{\sqcup_{j=1}^s X_j}$ is a concentric
family of circles in $\R^2$. Sometimes, we need to
use a disk region in the target: in such a case, we
can choose the region that does not contain the 
point $\infty \in S^2 = \R^2 \cup \{\infty\}$.
Note also that we can arrange $f$ in such a way that $f|_{S(f)}$
is an embedding by slightly perturbing $f$ near
$S(f)$ if necessary.
\end{proof}

Let us go on to the proof of the second theorem.

Let $f : M \to \R^2$ be a round fold map
of a closed orientable $3$--dimensional manifold.
By post-composing an
appropriate diffeomorphism of $\R^2$, we may assume that
$f(S(f))$ consists of the circles centered at
the origin with radii $1, 2, \ldots, t$.
Recall that, for $r > 0$,
$C_r$ denotes 
the circle of radius $r$ centered at the origin in $\R^2$.
We also put, for $0 < a < b$,
$$C_{[a, b]} = \{(x, y) \in \R^2\,|\, a \leq \sqrt{x^2+y^2} \leq b\}.$$

We can observe that $f^{-1}(C_{k-(1/2)})$ is a finite disjoint 
union of tori for each $k = 1, 2, \ldots, t$, since
$M$ is orientable.
Let $K$ be the closure of a component of 
$$M \setminus \left(\sqcup_{k=1}^t f^{-1}(C_{k-(1/2)})\right)$$
such that $f(K) \subset C_{[k-(1/2), k+(1/2)]}$.
Let $p_K : K \to S^1$ be the composition of $f|_K : 
K \to C_{[k-(1/2), k+(1/2)]}$ and the radial projection
$C_{[k-(1/2), k+(1/2)]} \to S^1$.
We see easily that $p_K$ is a submersion and 
hence is a locally trivial fibration. The fiber 
is a disjoint union of copies of $D^2$, $A$ and $P$.
Since $f|_{S(f)}$ is an embedding and $K$ is connected,
the fiber is diffeomorphic to $D^2$, $P$, or a finite
disjoint union of copies of $A$. 
If the fiber is diffeomorphic to $D^2$, then $K$
is diffeomorphic to $D^2 \times S^1$, since $K$
is an orientable $3$--dimensional manifold.
If the fiber is diffeomorphic to $P$, then
$K$ is diffeomorphic either to $P \times S^1$
or a non-trivial $P$--bundle over $S^1$
(see the proof of \cite[Lemma~2.4]{Sa1}).

Suppose that $K$ is a non-trivial $P$--bundle over $S^1$
and that $C_k \subset f(S(f))$ is inward-directed.
If $k=1$, then this leads to a contradiction, since
$f$ is a trivial bundle over the innermost region of
$\R^2 \setminus f(S(f))$. If $k > 1$, then
a component of $f^{-1}(C_{[k-(3/2), k-(1/2)]})$
adjacent to $K$ is either a non-trivial $P$--bundle
over $S^1$, or a non-trivial $(A \sqcup A)$--bundle
over $S^1$, where $A \sqcup A$ is the disjoint union
of two copies of $A$ and the monodromy for the latter
bundle interchanges the two components of $A \sqcup A$.
In the former case, $C_{k-1} \subset f(S(f))$
is outward-directed. In the latter case, we can repeat
the argument toward inner components to find an
outward-directed component.

If $C_k \subset f(S(f))$ is outward-directed, then
we can also find an inward-directed component outside
of $C_k$, since the outermost component corresponds to
$D^2 \times S^1$.

Thus we have proved the following.

\begin{lem}\label{lem:Pb}
Let $f : M \to \R^2$ be a round fold map
of a closed orientable $3$--dimensional manifold
such that
$f(S(f)) = \cup_{k=1}^t C_k$.
If $f$ is directed, then the closure
of a component of 
$$M \setminus \left(\sqcup_{k=1}^t f^{-1}(C_{k-(1/2)})\right)$$
is never diffeomorphic to the non-trivial $P$--bundle
over $S^1$.
\end{lem}

\begin{proof}[Proof of Theorem~\textup{\ref{thm2}}]
First, suppose that there exists a directed
round fold map $f : M \to \R^2$. By post-composing an
appropriate diffeomorphism of $\R^2$, we may assume that
the components of $f(S(f))$ are circles centered at
the origin with radii $1, 2, \ldots, t$.
Then the disjoint union of tori
$\sqcup_{k=1}^t f^{-1}(C_{k-(1/2)})$
decomposes $M$ into a union of copies of $P \times S^1$,
$A \times S^1$,
and $D^2 \times S^1$ attached along their torus boundaries.
Note that by Lemma~\ref{lem:Pb},
a non-trivial $P$--bundle over $S^1$
does not appear, since $f$ is directed. 
Furthermore, we can ignore the components diffeomorphic
to $A \times S^1 \cong T^2 \times [-1, 1]$ for
obtaining a decomposition of $M$.

As $f$ is directed and $M$ is connected, 
we see that the components 
diffeomorphic to $D^2 \times S^1$ are the outermost
component $f^{-1}(C_{[t-(1/2), t+(1/2)]})$
together with the components of the 
innermost part $f^{-1}(C_{[0, 1/2]})$: no other components
are diffeomorphic to $D^2 \times S^1$.
Then, we see easily that the corresponding graph describing
this decomposition of $M$ into copies of
$D^2 \times S^1$ and $P \times S^1$
is a tree, as the number of components of regular fibers
strictly increases toward the central region.

Conversely, suppose that
the graph describing the decomposition of
$M$ into copies of $P \times S^1$ and $D^2 \times S^1$ is a tree.
By inserting pieces diffeomorphic to $A \times S^1$ if necessary,
we may assume that the decomposition is of a plumbing type.
Then, the graph $\Gamma$ describing this new decomposition is also a tree.
Note that then $\Gamma$ has at least one vertex of degree one.
Let $s$ denote the number of vertices of $\Gamma$.
We label the vertices by $\{1, 2, \ldots, s\}$
in such a way that
\begin{enumerate}
\item the labeling gives a one-to-one correspondence
between the set of vertices and the set $\{1, 2, \ldots, s\}$,
\item a vertex of degree one has the label $s$,
\item for each $j \in \{1, 2, \ldots, s\}$,
the vertices of labels $\geq j$ together with
the edges connecting them constitute a connected subgraph of $\Gamma$.
\end{enumerate}
This is possible, since $\Gamma$ is a tree with only vertices
of degrees one, two or three.

Then, we follow the procedure as in the proof of
Theorem~\ref{thm1} for constructing a round fold
map on $M$, except for the components corresponding to
vertices of degree one whose label is different from $s$.
Note that in the process described
in the proof of \cite[Theorem~3.1]{Sa1}, we do not
need to use $h_1 : S^1 \times S^1 \times [-1, 1]
\to \R$ in our situation. Furthermore, when we use
$h_2$, we make sure that the corresponding image is
contained in $C_{[0, s]}$. Finally, for the components
corresponding to vertices of degree one with label $< s$,
we just consider the projection $D^2 \times S^1
\to D^2$, where the target $D^2$ should be enlarged
depending on the label. This matches with the
construction for the adjacent component.

Now, it is not difficult to see that the resulting
map $f : M \to \R^2$ is a directed round fold map.
This completes the proof.
\end{proof}

We do not know how to generalize Theorems~\ref{thm1}
and \ref{thm2} for non-orientable $3$--manifolds.
We also do not know how to classify the right-left
equivalence classes of (directed) round fold
maps on a given $3$--manifold (see a certain classification
result for simple stable maps given in \cite{Sa1}).

\section{Corollaries and examples}\label{section4}

In this section, we give some corollaries 
of our main theorems. We also show
that the class of $3$--manifolds that admit
directed round fold maps is strictly smaller than
that of all graph $3$--manifolds.

\begin{cor}
Let $M$ be a closed connected orientable graph $3$--manifold.
If $H_1(M; \Q) = 0$, then it admits a directed
round fold map into $\R^2$.
\end{cor}

\begin{proof}
Let $G$ be the graph corresponding to a
decomposition of $M$ into $P \times S^1$
and $D^2 \times S^1$ as described in Lemma~\ref{lem1}.
Then, we can naturally construct a
continuous map $\gamma : M \to G$
in such a way that for each piece, the complement
of a small collar neighborhood of the boundary
is mapped to the corresponding vertex.
Then, we can show that $\gamma$
induces a surjection $\gamma_* : \pi_1(M) \to
\pi_1(G)$. Since $H_1(M; \Q) = 0$, we see that
$G$ is a tree. Then, the result follows from
Theorem~\textup{\ref{thm2}}.
\end{proof}

Since every closed orientable Seifert $3$--manifold
over the $2$--sphere admits
a decomposition into
a union of a finite number of copies of $P \times S^1$
and a solid torus such that the corresponding graph is a tree,
we have the following.

\begin{cor}
Every closed orientable Seifert $3$--manifold over $S^2$ admits
a directed round fold map into $\R^2$. 
\end{cor}

By virtue of the realization result due to \cite{BD},
as a corollary, we see that
every linking form can be realized as that
of a $3$--manifold admitting a directed round fold
map into $\R^2$. Thus, the linking form
cannot detect the non-existence 
of a directed round fold map.

On the other hand, as to the cohomology ring, we have the
following.

\begin{cor}
If a closed orientable $3$--manifold $M$
admits a directed round fold map into $\R^2$,
then for every pair $\xi, \eta \in H^1(M; \Q)$,
their cup product $\xi \smile \eta$ vanishes in $H^2(M; \Q)$.
\end{cor}

The above corollary follows from \cite[Theorem~5.2]{DH}.

Thus, for example, for every closed orientable
surface $\Sigma$ of genus $\geq 1$, the $3$--manifold
$\Sigma \times S^1$ never admits a directed
round fold map into $\R^2$, although it is a graph manifold.

Let $M$ be a closed connected orientable $3$-dimensional
manifold. If $M$ is a graph manifold, then it
can be represented by a \emph{plumbing graph} 
whose vertices and edges
have certain weights \cite{N}. Such a graph is not unique
as the decompositions of $M$ into $S^1$--bundles over
surfaces are not unique. However, we have the notion of
a \emph{normal form}, and then we have the existence and uniqueness
of such a normal form for a given closed connected orientable
graph manifold \cite[Theorem~4.1]{N}.

The following lemma can be proved by following the
proof of \cite[Theorem~4.1]{N}.

\begin{lem}
Let $M$ be a closed connected orientable graph $3$--manifold.
If it can be decomposed into
a union of finite numbers of copies of $P \times S^1$
and a solid torus in such a way that the corresponding graph is 
a tree,
then its normal form plumbing graph is a finite disjoint
union of trees.
\end{lem}

As a corollary, we have the following.

\begin{cor}
Let $M$ be a closed connected orientable graph $3$--manifold
whose normal form plumbing graph contains a loop.
Then, $M$ admits a round fold map into $\R^2$ but does not
admit a directed round fold map into $\R^2$.
\end{cor}

For example, some torus bundles over $S^1$ as described
in \cite[Theorem~6.1]{N} satisfy the assumption
of the above corollary. (More precisely, those torus bundles over $S^1$
whose monodromy matrix has trace $\geq 3$ or $\leq -3$,
give such examples.)

\section*{Acknowledgment}\label{ack}
The authors would like to thank Professor Yuya Koda
for stimulating discussions which motivated the
theme of this paper.
This work was supported by JSPS KAKENHI Grant Number 
JP17H06128.


\end{document}